\documentclass{amsart}
\usepackage{amsmath}
\usepackage{amssymb}
\usepackage{graphicx}
\input xy
\xyoption{all}
\newtheorem{Lemma}{Lemma}[section]
\newtheorem{theorem}[Lemma]{Theorem}

\newtheorem{proposition}[Lemma]{Proposition}
\newtheorem{corollary}[Lemma]{Corollary}

\newtheorem{remark}[Lemma]{Remark}

\newcommand{\Core}{\mbox{\rm Core}}

\newcommand{\Char}{\mbox{\rm char}}

\newcommand{\GL}{\mbox{\rm GL}}
\newcommand{\SL}{\mbox{\rm SL}}

\newcommand{\Z}{\mathbb{Z}}

\newcommand{\Q}{\mathbb{Q}}

\newcommand{\F}{\mathbb{F}}

\begin{document}
\title[Free subgroups in almost subnormal subgroups] {Free subgroups in almost subnormal subgroups\\ of general skew linear groups}
\author{Nguyen Kim Ngoc}\author[Mai Hoang Bien]{Mai Hoang Bien}\author[Bui Xuan Hai]{Bui Xuan Hai}
 
\address{Nguyen Kim Ngoc}
 \address{University of Science, VNU-HCM, Vietnam}
 \email{nkngoc1985@gmail.com}
 \address{Faculty of Mathematics and Computer Science, University of Science, VNU-HCM, 227 Nguyen Van Cu Str., Dist. 5, Ho Chi Minh City, Vietnam.}  
  
 \address{Mai Hoang Bien}
 \address{University of  Architecture, HCM, Vietnam}
 \email{ maihoangbien012@yahoo.com}
 \address{Department of Basic Sciences, University of  Architecture\\ 196 Pasteur Str., Dist. 1, Ho Chi Minh City, Vietnam} 
 
 \address{Bui Xuan Hai}
 \address{University of Science, VNU-HCM, Vietnam}
 \email{bxhai@hcmus.edu.vn}
 \address{Faculty of Mathematics and Computer Science, University of Science, VNU-HCM, 227 Nguyen Van Cu Str., Dist. 5, Ho Chi Minh City, Vietnam.}   


\maketitle
\begin{abstract} Let $D$ be a weakly locally finite division ring  and $n$ a positive integer. In this paper, we investigate the problem on the existence of non-cyclic free subgroups in non-central almost subnormal subgroups of the general linear group ${\rm GL}_n(D)$. Further, some applications of this fact are also investigated. In particular, all infinite finitely generated almost subnormal subgroups of ${\rm GL}_n(D)$ are described. 
\end{abstract}
\section{Introduction and preliminaries}
Let $G$ be a group and $H$ be a subgroup of $G$. According to Hartley \cite{Pa_Ha_89}, we say that $H$ is {\it almost subnormal} in $G$, and write $H$ asn $G$ for short, if there is a family of subgroups 
$$H=H_r\le H_{r-1}\le\cdots\le H_1=G$$
of $G$ such that for each $1<i\le r$, either $H_{i}$ is normal in $H_{i-1}$ or $H_i$ has finite index in $H_{i-1}$. We call such a series of subgroups an {\it almost normal series} in $G$. In this paper, we study the problem on the existence of non-cyclic free subgroups in almost subnormal subgroups of the general linear group ${\rm GL}_n(D)$ over a division ring $D$ and the related problems. 

The question on the existence of non-cyclic free subgroups in linear groups over a field was studied by Tits in \cite{tits}. The main theorems of Tits assert that in the characteristic $0$, every subgroup of the general linear group ${\rm GL}_n(F)$ over a field $F$ either contains a non-cyclic free subgroup or it is soluble-by-finite, and the same conclusion for finitely generated subgroups in the case of prime characteristic. This famous result of Tits is now often referred as Tits' Alternative. The question of whether Tits' Alternative would remain true for skew linear groups was posed by S. Bachsmuth at the Second International Conference on the Theory of Groups (see \cite[p. 736]{bach}). In \cite{lichtman_77}, Lichtman has proved that there exists a finitely generated group which is not soluble-by-finite and does not contain a non-cyclic free subgroup, but whose group ring over any field can be embedded in a division ring of quotients. Therefore, Tits' Alternative fails even for matrices of degree one, i.e. for $D^*={\rm GL}_1(D)$, where $D$ is a non-commutative division ring. In \cite{lichtman_77}, Lichtman remarked that it is not known whether the multiplicative group of a non-commutative division ring contains a non-cyclic free subgroup. In \cite{gon-ma_86}, Gon\c calves and Mandel posed more general question: whether a non-central subnormal subgroup of the multiplicative group of a division ring contains a non-cyclic free subgroup? This question was studied by several authors. Gon\c calves \cite{gon_84_1} proved that the multiplicative group $D^*$ of a division ring $D$ with center $F$  contains a non-cyclic free subgroup if $D$ is centrally finite, that is, $D$ is a finite dimensional vector space over $F$. The same result was obtained by Reichstein and Vonessen in \cite{rei-von} if $F$ is uncountable  and there exists a non-central element $a$ in $D$ which is algebraic over $F$. Later, Chiba \cite{chiba} proved the same result but without the assumption on the existence of such an element $a$ in $D$. In \cite{gon_84_2}, Gon\c calves proved that any non-central subnormal subgroup of $D^*$ contains a non-cyclic free subgroup provided $D$ is centrally finite. Recently, B. X. Hai and N. K. Ngoc \cite{hai-ngoc} proved the same result for weakly locally finite division rings. Recall that a division ring $D$ is called {\it weakly locally finite} if every finite subset of $D$ generates a centrally finite division subring. It was proved that every locally finite division ring is weakly locally finite, and there exist infinitely many weakly locally finite division rings that are not even algebraic over their centers (see \cite{hai-ngoc} and \cite{hbd}), so they are not locally finite. Logically, it is natural to carry over the  results above for subnormal subgroups of ${\GL}_1(D)$ to that of ${\rm GL}_n(D), n\geq 2$ (see \cite{Pa_GoSh_09, lichtman_87, shir-weh}). In the present paper, we investigate the question on the existence of free subgroups in almost subnormal subgroups of the group ${\rm GL}_n(D)$ with $n\geq 1$ and $D$ is a division ring unnecessarily commutative. 

Note that in \cite[Example 8]{Haz-Wad}, Hazrat and Wadsworth gave the examples of division rings whose multiplicative groups contains non-normal maximal subgroups of finite index. Hence, the existence of almost subnormal subgroups in $D^*:={\rm GL}_1(D)$ has no doubt. Concerning the group ${\rm GL}_n(D), n\geq 2$, we shall prove that if $D$ is infinite then every almost subnormal subgroup of  ${\rm GL}_n(D)$ is normal (see Theorem \ref{th8.2} in the text). However, we shall continue to use ``almost subnormal" instead of ``normal" to compare the results
with the corresponding ones in the case n = 1.

All symbols and notation we use in this paper are standard. In particular, if $A$ is a ring or a group, then $Z(A)$ denotes the center of $A$. If $D$ is a division ring with center $F$ and $S$ is a subset of $D$, then $F(S)$ denotes the division subring of $D$ generated by the set $F\cup S$. We say that $F(S)$ is the division subring of $D$ \textit{generated by $S$ over $F$.} Finally, $D':=[D^*,D^*]$ is the commutator subgroup of $D^*$. The following lemma which  will be used frequently in this paper, is almost evident, so we omit its proof.

\begin{Lemma}\label{l1.1} Let $H$ be an almost subnormal subgroup of a subgroup $G$. If $N$ is a subgroup of $G$ containing $H$, then $H$ is an almost subnormal subgroup of $N$.  
\end{Lemma} 

\section{Almost subnormal subgroups with generalized group identities}
Let $G$ be a group with center $Z(G)=\{\,a\in G\mid ab=ba \text{ for any } b\in G\,\}$. An expression $$w(x_1,x_2, \cdots, x_m)=a_1x_{i_1}^{\alpha_1}a_2x_{i_2}^{\alpha_2}\cdots a_tx_{i_t}^{\alpha_t}a_{t+1},$$ where $t,m$ are positive integers, $i_1,i_2, \cdots, i_t\in \{\,1,2,\cdots, m\,\}$, $a_1,a_2, \cdots, a_{t+1}\in G$ and $\alpha_1, \alpha_2,\cdots, \alpha_{t}\in \Z\backslash\{0\}$, is called a {\it generalized group monomial} over $G$ if whenever  $i_j=i_{j+1}$ with $\alpha_j\alpha_{j+1}<0$ ($j\in\{1,2,\cdots, t-1\}$), then $a_j\not\in Z(G)$ (see \cite{Pa_To_85}). Moreover, if one has $i_j\ne i_{j+1}$ whenever $\alpha_j\alpha_{j+1}<0$, then we say that $w$ is a \textit{strict generalized group monomial} over $G$. If $G=\{1\}$, then we simply call $w$ a {\it group monomial}. It is clear that a group monomial is a strict generalized group monomial.

Let $H$ be a subgroup of $G$. We say that $H$ satisfies the {\it generalized group identity} $w(x_1,x_2, \cdots, x_m)=1$ if $w(c_1,c_2, \cdots, c_m)=1$ for any $c_1,c_2,\cdots, c_m\in H$. In particular, we say that $H$ satisfies a {\it group identity} (resp. \textit{strict generalized group identity}) $w=1$ if $w$ is a group monomial (resp. strict generalized group monomial) and $w(c_1,c_2, \cdots, c_m)=1$ for any $c_1,c_2,\cdots, c_m\in H$.

In this section, we prove some properties of almost subnormal subgroups with generalized group identities we need for the next study. We begin with the following  useful lemma. 
\begin{Lemma}\label{l2.1}
	Let $G$ be a group  and assume that $H$ is a non-central almost subnormal subgroup of $G$. If $H$ satisfies a generalized group identity over $G$, then so does $G$.
\end{Lemma}
\begin{proof}  Assume that $H$ is a non-central almost subnormal subgroup of $G$ satisfying a generalized group identity over $G$. Let $H=H_r\le H_{r-1}\le\cdots\le H_1=G$ be an almost normal series in $G$. To prove the lemma, that is, to prove that $H_1$ satisfies a generalized group identity, it suffices to prove that $H_{r-1}$ satisfies a generalized group identity over $G$.
	
	Let $w(x_1,x_2\cdots,x_m)=a_1x_{i_1}^{n_1}a_2x_{i_2}^{n_2}\cdots a_tx_{i_t}^{n_t}a_{t+1}$ be a generalized group identity of $H$ over $G$. By replacing $x_i=y_iy_{i+m}$, we get $$u(y_1,y_2,\cdots,y_{2m})=w(y_1y_{1+m},y_2y_{2+m},\cdots,y_my_{2m})=b_1y_{j_1}^{\delta_1}b_2y_{j_2}^{\delta_2}\cdots b_sy_{j_s}^{\delta_s}b_{s+1},$$ where $\delta_i\in\{-1,1\}$, which is also a generalized group identity of $H$. Hence, without loss of generality, we can assume that in $w$, the powers $n_i\in \{-1,1\}$ for any $1\le i\le t$. There are two cases to examine:
\bigskip
	
\noindent
{\it Case 1.  $H_r$ has finite index $k$ in $H_{r-1}$:}	 
		
Then, for any $c_1,c_2,\cdots,c_m\in H_{r-1}$, we have $c_1^k, c_2^k, \cdots, c_m^k\in H_r$. By hypothesis,  $w(c_1^k,c_2^k,\cdots,c_m^k)=1$. But this means that  $H_{r-1}$ satisfies the identity $w(x_1^k,x_2^k,\cdots,x_m^k)=1$.

\bigskip		
\noindent
{\it Case 2. $H_r$ is normal in $H_{r-1}$:} 

Since $H=H_r$ is non-central, there exists a non-central element $a\in H_r$. Replacing $x_j$ by $x_jax_j^{-1}$ for any $1\le j\le m$, we get $$w_1(x_1,x_2,\cdots,x_m)=w(x_1ax_1^{-1}, x_2ax_2^{-1},\cdots, x_max_m^{-1}),$$ 
which is a generalized group monomial over $G$ by \cite[Lemma 3.2]{Pa_Bi_15}. Since $H_r$ is normal in $H_{r-1}, c_iac_i^{-1}\in H_r$ for any $c_i\in H_{r-1}$. Therefore,  $$w_1(c_1, c_2, \cdots, c_m)=w(c_1ac_1^{-1},c_2ac_2^{-1},\cdots,c_mac_m^{-1})=1,$$ 
for any $c_1, c_2, \cdots, c_m\in H_{r-1}$,
which shows that $w_1=1$ is a generalized group identity of $H_{r-1}$.		
Therefore, the proof of the lemma is now completed.	
\end{proof}

From Theorems 1, 2 in \cite{Pa_GoMi_82}, it follows that for any division ring $D$ with infinite center $F$, if $\GL_n(D)$ satisfies a generalized group identity, then $n=1$ and $D=F$. Recently, it was proved that this result remains true if one replaces $\GL_n(D)$ by any its subnormal subgroup \cite[Theorem 1.1]{Pa_BiKiRa_?}. Lemma \ref{l2.1} gives us the possibility to get the following strong result.
\begin{theorem}\label{t2.2}
	Let $D$ be a division ring with infinite center $F$ and assume that $G$ is an almost subnormal subgroup of $\GL_n(D)$. If $G$ satisfies a generalized group identity over $\GL_n(D)$, then $G$ is central. 
\end{theorem}
\begin{proof}
	If $G$ is non-central, then by Lemma~\ref{l2.1}, $\GL_n(D)$ satisfies a generalized group identity. In  view of \cite{Pa_GoMi_82}, $n=1$ and $D$ is commutative, so  $G$ is central, a contradiction. Thus, $G$ is central.
\end{proof}

\begin{corollary}\label{c2.3} Let $D$ be a division ring with infinite center $F$ and assume that $G$ is an almost subnormal subgroup of $\GL_n(D)$. If $G$ is abelian, then $G$ is central. 
\end{corollary}
\begin{proof}
	If $G$ is abelian, then $G$ satisfies the group identity $xyx^{-1}y^{-1}=1$. So, by Theorem~\ref{t2.2}, $G$ is central.
\end{proof}

\section{Almost subnormal subgroups of $\GL_n(D)$ is normal}

Recall that a field $K$ is called {\it locally finite} if every subfield generated by finitely many elements of $K$ is finite. Hence, $K$ is locally finite if and only if its prime subfield $P$  is a finite field and $K$ is algebraic over $P$. In \cite{shir-weh}, there is the following theorem.
\bigskip

\noindent
{\bf Theorem A.}
\textit{Let $D$ be a division ring that is not a locally finite field and let $n>1$ be an integer. If $N$ is non-central normal subgroup of ${\rm GL}_n(D)$ then $N$ contains a non-cyclic free subgroup.}

The aim of this section is to prove  that if $D$ is an infinite division ring  and $n\ge 2$, then every almost subnormal subgroup of $\GL_n(D)$ is normal in $\GL_n(D)$. Hence, according to {\bf Theorem  A}, the problem on the existence of non-cyclic free subgroups in almost subnormal subgroups of general skew linear groups reduced to that in skew linear groups of degree $1$. To prove this fact, we need the following some results. 
\begin{theorem}\label{th.2} Let $D$ be a division ring, $n$ a natural number and $N$ a non-central subgroup of $GL_n(D)$. Suppose that either $n\geq 3$ or that $n=2$ but that $D$ contains at least four elements. If $xNx^{-1}\subseteq N$ for any $x\in \SL_n(D)$, then $N$ contains $\SL_n(D)$. In particular, a non-central subgroup of ${\rm GL}_n(D)$ is normal in ${\rm GL}_n(D)$ if and only if it contains ${\rm SL}_n(D)$.
\end{theorem}
\begin{proof} The proof follows from Theorem 4.7 and Theorem 4.9 in \cite{artin}.
\end{proof}

\begin{Lemma}\label{th.3} Let $D$ be a division ring and $n>1$. Then, the special linear group $\SL_n(D)$ satisfies a group identity if and only if $D$ is finite.
\end{Lemma}
\begin{proof}
	If $D$ is finite, then $\SL_n(D)$ is finite, so  $\SL_n(D)$ satisfies a group identity. Assume that $D$ is infinite. Let $K$ be a maximal subfield of $D$. If $K$ is finite, then $\dim_FD<\infty$ by \cite[(15.8)]{Bo_La_91} which implies that $D$ is finite, a contradiction. Therefore, $K$ is infinite. Suppose that $w(x_1,x_2,\cdots,x_m)=1$ is a group identity of $\SL_n(D)$. Then, $\GL_n(D)$ satisfies the group identity $$w_1(y_1,y_2,\cdots,y_{2m})=w(y_1y_2y_1^{-1}y_2^{-1},y_3y_4y_3^{-1}y_4^{-1},\cdots,y_{2m-1}y_{2m}y_{2m-1}^{-1}y_{2m}^{-1})=1.$$ In  view of Theorem~\ref{t2.2}, $\GL_n(D)$ is central, a contradiction.
\end{proof}


Let $G$ be a group and $H$ a subgroup of $G$. Denote by $\Core_G(H)$ the \textit{core} of $H$ in $G$, that is,  $$\Core_G(H)=\bigcap_{x\in G}~~xHx^{-1}.$$ It is well known that $\Core_G(H)$ is the largest normal subgroup of $G$ which is contained in $H$. Moreover, if $H$ is of finite index in $G$, then so is $\Core_G(H)$.
\bigskip

The following theorem is the main result of this section.
\begin{theorem}\label{th8.2}
	Let $D$ be an infinite division ring  and $n\ge 2$. Assume that $N$ is a non-central subgroup of $G:=\GL_n(D)$. The following conditions are equivalent:
	\begin{enumerate}
		\item $N$ is almost subnormal in $G$.
		\item $N$ is subnormal in $G$.
		\item $N$ is normal in $G$.
		\item $N$ contains $\SL_n(D)$.
	\end{enumerate}
\end{theorem}
\begin{proof}
	The implications 
	(4) $\Rightarrow$ (3) $\Rightarrow$ (2) $\Rightarrow$ (1) are trivial. Now we will show that (1) implies (4). Assume that $N$ is a non-central almost subnormal subgroup of $G$ and $$N=N_r\le N_{r-1}\le\cdots\le N_1\le N_0=G$$ is
	an almost normal series of $N$ in $G$. We shall prove that $N_i$ is normal in $G$ for all $1 \leq i\leq r$ by induction on $i$.  Assume that  $N_1$ has  finite index in $G$. Then, $\Core_G(N_1)$ is a normal subgroup, say  of finite index $m$ in $G$ and it is contained in $N_1$. If $\Core_G(N_1)$ is central, then  $x^{m}y^{m}x^{-m}y^{-m}=1$ for any $x, y\in G$. In view of Lemma~\ref{th.3}, $D$ is finite, a contradiction.  Thus, $\Core_G(N_1)$ is non-central normal subgroup of $G$. By Theorem \ref{th.2}, $\SL_n(D)\subseteq \Core_G(N_1)\subseteq N_1$. 
	
	Assume that $j>1$ and $N_j$ contains $\SL_n(D)$. We must prove that $N_{j+1}$ also contains  $\SL_n(D)$. Indeed, assume that $N_{j+1}$ has finite index in $N_j$. Then, the subgroup $\Core_{N_{j}}(N_{j+1})$ is normal in $N_j$ of finite index, say $k$, and it is contained in $N_{j+1}$. Assume that $\Core_{N_{j}}(N_{j+1})$ is central. Then, $x^ky^kx^{-k}y^{-k}=1$ for any $x, y\in N_{j}$. In particular, $\SL_n(D)$ satisfies the group identity $x^ky^kx^{-k}y^{-k}=1$. In view of Lemma~\ref{th.3}, $D$ is finite, a contradiction. Hence, $\Core_{N_{j}}(N_{j+1})$ is non-central normal subgroup of $N_j$. Therefore, $x\Core_{N_{j}}(N_{j+1})x^{-1}\subseteq \Core_{N_{j}}(N_{j+1})$ for any $x\in \SL_n(D)\subseteq N_j$.  In view of  Theorem \ref{th.2}, $\Core_{N_{j}}(N_{j+1})$ contains ${\rm SL}_n(D)$  and so does $N_{j+1}$.
	
	The implication (1) $\Rightarrow$ (4) is proved, and so the proof of the theorem is now  complete.
\end{proof}

\begin{remark}
	{\rm Theorem \ref{th8.2} no longer holds if $D$ is a finite field. Indeed, let $D=\F_q$ be a finite field with $q$ elements and consider the projective special linear group ${\rm PSL}(n, q)$ which is different from the groups ${\rm PSL}(2, 2)$ and ${\rm PSL}(2, 3)$. Then, it is well-known that ${\rm PSL}(n, q)$ is a simple group. Assume that $k$ is a prime divisor of $|{\rm PSL}(n, q)|$ and $H$ is the inverse image of a subgroup of order $k$ in ${\rm PSL}(n, q)$ via the natural homomorphism
		$${\rm SL}(n, q)\longrightarrow {\rm PSL}(n, q).$$
		
		Then, $H$ is a non-central proper subgroup of ${\rm SL}(n, q)$. By Theorem \ref{th.2}, $H$ is not normal in ${\rm GL}(n, q)$. Hence, $H$ is an almost subnormal subgroup of ${\rm GL}(n, q)$ which is not normal in ${\rm GL}(n, q)$. }
	
\end{remark}

\section{Non-cyclic free subgroups in non-commutative division rings}

As we have mentioned in the Introduction, almost subnormal subgroups that are not subnormal exist in the multiplicative group of a division ring. The aim of this section is to show that if a non-commutative division ring $D$ is weakly locally finite, then every non-central almost subnormal subgroup of $D^*$ contains a non-cyclic free subgroup. Recall that a division ring $D$ is {\it weakly locally finite} if every finite subset in $D$ generates a centrally finite division subring in $D$. Some basic properties and the existence of non-cyclic free subgroups in weakly locally finite division rings can be seen in \cite{Pa_DeBiHa_12} and \cite{hai-ngoc}. The following lemma is useful for our next study.  

\begin{Lemma}\label{l4.2}
	Let $D$ be a non-commutative weakly locally finite division ring with center $F$ and assume that $G$ is an almost subnormal subgroup of $D^*$. If $G$ satisfies a generalized group identity, then $G$ is central. In particular, If $G$ is abelian, then $G$ is central.
\end{Lemma}
\begin{proof}
	Assume that $G$ is non-central and $G$ satisfies some generalized group identity $w(x_1,x_2,\cdots,x_m)=1$. By Lemma \ref{l2.1}, $D^*$ satisfies some  generalized group identity $w'(x_1,\ldots, x_m)=a_1x_1^{t_1}a_2x_2^{t_2}\cdots a_mx_m^{t_m}a_{m+1}=1$. Let $x,y\in D$ such that $xy\neq yx$. Consider  the division subring  $D_1$ of $D$ generated by $x,y$ and all $a_i$. Then,  $D_1$ is non-commutative and centrally finite. Since $a_i\in D_1^*\subseteq D^*,  w'=1$ is a generalized group identity of $D_1^*$. In view of  Theorem~\ref{t2.2}, $D_1$ is commutative, a contradiction. Hence, $G$ is central. 
\end{proof}

\begin{theorem}\label{t4.3}
	Let $D$ be a weakly locally finite division ring with center $F$. Then, every non-central almost subnormal subgroup of $D^*$ contains a non-cyclic free subgroup.
\end{theorem}
\begin{proof}
	Assume that $G$ is a non-central almost subnormal subgroup of $D^*$. By Lemma~\ref{l4.2}, $G$ is non-abelian. Hence, there exist $a,b\in G$ such that $ab\neq ba$. Denote by $D_1$  a division subring of $D$ generated by $a,b$. Then, $D_1$ is centrally finite. By Lemma \ref{l1.1}, $N=G\cap D_1^*$ is an almost subnormal subgroup of $D_1^*$.  We claim that $N$ contains a non-cyclic free subgroup. Indeed, if this is not the case then, by \cite[Theorem 2.21]{Pa_HaMaMo_14}, $N$ satisfies a group identity $w(x_1,x_2,\ldots,x_m)=1$.  Observe that the center $F_1$ of $D_1$ is infinite, so by Theorem~\ref{t2.2}, $N$ is central. In particular, $ab=ba$ which is a contradiction. Thus, the claim is proved, and this implies that $G$ contains a non-cyclic free subgroup.
\end{proof}

Now combining Theorem~\ref{t4.3} and Theorem A, we summary the results on the existence of non-cyclic free subgroups in almost subnormal subgroups of the general linear group over a weakly locally finite division ring.

\begin{theorem}\label{t4.3'} Let $D$ be a weakly locally finite division ring, $n$ a natural number and $N$ a non-central almost subnormal subgroup of $\GL_n(D)$. Then, $N$ contains a non-cyclic free subgroup if one of the following conditions satisfies:
	\begin{enumerate}
		\item $D$ is non-commutative;
		\item $n=1$;
		\item $n\geq 2$ and $D$ is not a locally finite field.
	\end{enumerate}
\end{theorem}

The following proposition extends \cite[Corollary 3.4]{gon-ma_86}. 
\begin{proposition}\label{p7.4} Let $D$ be a division ring with center $F$ and $G$ be an almost subnormal subgroup of $D^*$. If $G\backslash F$ contains a torsion element, then $G$ contains a non-cyclic free subgroup. 
\end{proposition}
\begin{proof} Assume that $a\in G\backslash F$ such that $a^n=1$ for some positive integer $n$. According to \cite[Proposition 2.1]{Pa_BiDu_14}, there exists a centrally finite division ring $D_1$ such that $a\not\in F_1=Z(D_1)$. Using the same argument in the proof of Theorem~\ref{t4.3}, we see that $M=G\cap D_1^*$ is an almost subnormal subgroup of $D^*_1$. Since $a\in M$, the subgroup $M$ is non-central. By Theorem~\ref{t4.3}, $M$ contains a non-cyclic free subgroup, so does $G$.
\end{proof}	

\begin{theorem}\label{t4.4}
	 Let $D$ be a weakly locally finite division ring with center $F$ and $G$ be an almost subnormal subgroup of $D^*$. If $G$ is soluble-by-periodic, then $G$ is central.
\end{theorem}
\begin{proof}
	Assume that $G$ is non-central. Let $H$ be a soluble normal subgroup of $G$ such that $G/H$ is periodic. Then, $H$ is almost subnormal  in $D^*$. Since $H$ is soluble, $H$ is central by Lemma~\ref{l4.2}. Hence, for any $x\in G$, there exists a natural number $n_x$ such that $x^{n_x}\in H\subseteq F$. But this is impossible since $G$ contains a non-cyclic free subgroup by Theorem~\ref{t4.3}.
\end{proof}
Note that previously in \cite[Propositon 1]{Pa_Li_78}, Lichtman proved the same result as in Theorem \ref{t4.4} for normal subgroups in centrally finite division rings. The class of weakly locally finite division rings we consider in Theorem \ref{t4.4} is very large. Indeed, in \cite{hai-ngoc}, it was indicated that this class  strictly contains the class of locally finite division rings. Recently, in \cite{hbd}, we have constructed infinitely many examples of weakly locally finite division rings that are not even algebraic over the center.
 
Now, let $D$ be a centrally finite division ring. The following theorem gives useful characterization of a subgroup of $D^*$ that contains no non-cyclic free subgroups. 
\begin{theorem}\label{t4.5}	Let $D$ be a centrally finite division ring and $G$ be a subgroup of $D^*$. The following conditions are equivalent:
	\begin{enumerate}
		\item $G$ contains no non-cyclic free subgroups.
		\item $G$ is soluble-by-finite.
		\item $G$ is abelian-by-finite.
		\item $G$ satisfies a group identity.
		\item $G$ contains a soluble subgroup of finite index.
		\item $G$ contains an abelian subgroup of finite index.
		\item $G$ satisfies a strict generalized group identity.
	\end{enumerate}
\end{theorem}
\begin{proof}
	For (1) $\Leftrightarrow$ (2)$\Leftrightarrow$ (3)$\Leftrightarrow$ (4) see \cite[Theorem 2.21]{Pa_HaMaMo_14}. The implications (2) $\Rightarrow$ (5), (3) $\Rightarrow$ (6), (6) $\Rightarrow$ (5), and (4) $\Rightarrow$ (7) are trivial.
	
	To prove (5) $\Rightarrow$ (4), assume that $G$ contains a soluble subgroup $H$ of finite index $[G:H]=m$. Since $H$ is soluble, $H$ satisfies a group identity $w(x_1,x_2,\cdots,x_n)=~1$. Then,  $w(c_1^m,c_2^m,\cdots,c_n^m)=1$ for any $c_i\in G$, so  $G$ satisfies a group identity $w(x_1^m,x_2^m,\cdots,x_n^m)=1$.
	
	Finally, the equivalence (5) $\Leftrightarrow$ (7) follows from \cite[Theorem 1]{Pa_To_85}. The proof of a theorem is now complete.
\end{proof}
\bigskip

In \cite{Pa_Li_78}, Lichtman have shown that for a normal subgroup $G$ of $D^*$, if there exists a non-abelian nilpotent-by-finite subgroup in $G$, then $G$ contains a non-cyclic free subgroup. Recently, Gon\c calves and Passman gave another proof and an explicit construction of  non-cyclic free subgroups (see \cite{Pa_GoPa_15}). In the following theorem, we consider the case when $D$ is algebraic over its center and generalize this result for almost subnormal subgroups.   
\begin{theorem}\label{t8.1}
	Let $D$ be a  division ring algebraic over its center $F$ and assume that $G$ is an almost subnormal subgroup of $D^*$. If $G$ contains a non-abelian nilpotent-by-finite subgroup, then $G$ contains a non-cyclic free subgroup. 
\end{theorem}
\begin{proof}
	Let $N$ be a non-abelian nilpotent-by-finite subgroup of $G$. Then, there exists a nilpotent normal subgroup $A$ of $N$ such that $[N:A]=m$.
	
	\textit{Case 1. $A$ is non-abelian:}
		
Since $A$ is nilpotent, there exist $x, y\in A$ such that $$1\neq z=y^{-1}x^{-1}yx, zx=xz,zy=yz.$$ Let $D_1$ be a division subring of $D$ generated by ${x, y}$. By  \cite[Lemma 1]{Pa_Li_78}, $D_1$ is centrally finite. Using the same argument as in the proof of Theorem~\ref{t4.3}, we can conclude that  $M=G\cap D^*_1$ is a non-abelian almost subnormal subgroup in $D^*_1$. By Theorem \ref{t4.3}, $M$ contains a non-cyclic free subgroup. 
	
	\textit{Case 2. $A$ is abelian:} 
	
Let $D_1$ be a division subring of $D$ generated by $F$ and $N$, and $F_1=Z(D_1)$. Since $[N:A]=m$, $D_1$ is finite dimensional over subfield $F(A)$. By \cite{Pa_BeDrSh_13}, $D_1$ is centrally finite. Using the same argument as in the proof of Theorem~\ref{t4.3}, we conclude that   $M=G\cap D^*_1$ is a non-abelian almost subnormal subgroup in $D^*_1$. By Theorem \ref{t4.3}, $M$ contains a non-cyclic free subgroup. 
\end{proof}

\section{Finitely generated almost subnormal subgroups of $\GL_n(D)$}

In this section, we  investigate finitely generated subgroups of $\GL_n(D)$ with some additional conditions. Recall that if $D=F$ is a field then  Tits' Alternative asserts that every finitely generated subgroup of $\GL_n(F)$ either contains a non-cyclic free subgroup or it is soluble-by-finite. Now, assume that  $G$ is a finitely generated subgroup of $\GL_n(D)$, where $D$ is a non-commutative division ring. It was shown in \cite{Pa_MaMaYa_00} that if $D$ is centrally finite and $G$ is a subnormal subgroup in $\GL_n(D)$, then $G$ is central. In the case when $n=1$, it was proved in \cite[Theorem~ 2.5]{Pa_HaDeBi_12} that if $D$ is of type $2$, then there are no finitely generated subgroups of $D^*$ containing the center $F^*$. Recall that a division ring $D$ is said to be of {\it type $2$} if the division subring $F(x,y)$ of $D$ generated over its center $F$ by any two elements $x,y\in D$ is a finite-dimensional vector space over $F$. The aim of this section is to carry over these results  for  almost subnormal subgroups of ${\rm GL}_n(D)$, where $D$ is a weakly locally finite division ring. Recall that 
Theorem \ref{th8.2}  implies that any almost subnormal subgroup of $\GL_n(D)$ is normal if $n\geq 2$, but we
shall continue to use ``almost subnormal" instead of ``normal" to compare the results
with the corresponding ones in the case n = 1.

The following result, which is an easy consequence of Theorem~\ref{t2.2} gives the characterization of finite almost subnormal subgroups in division rings.
\begin{Lemma}\label{l6.1}
	Let $D$ be a division ring with center $F$ and assume that $G$ is an almost subnormal subgroup of $D^*$. If $G$ is finite, then $G$ is central.
\end{Lemma}
\begin{proof}
	Let $D_1=F(G)$ be the division subring of $D$ generated by $G$ over $F$. By Lemma \ref{l1.1}, $G$ is almost subnormal in $D_1^*$. If $F$ is finite, then $D_1$ is a field. In particular, $G$ is abelian, so by Theorem ~\ref{t4.4}, $G$ is central. If $F$ is infinite, then so is the center of $D_1$. Hence, in view of Theorem~\ref{t2.2}, $G$ is central. 
\end{proof}

\bigskip

\begin{remark}\label{r6.2}{\rm
	Let $H=\langle a_1,a_2,\cdots,a_m\rangle$ be a finitely generated subgroup of $\GL_n(D)$, where $D$ is a division ring.  Denote by $S$ the set of all entries of all matrices $a_i,a^{-1}_i$, and by $R$ the subring of $D$ generated by $S$. Then, $H$ is contained in $\GL_n(R)$. In particular, $H$ is contained in $\GL_n(D_1)$, where $D_1$ is the division subring of $D$ generated by $S$. This fact will be used frequently in this section.  }
\end{remark}

Let $G$ be a subgroup of $\GL_n(F)$, where $F$ is a field. Suppose that $a\in \GL_n(F)$. It is easy to see that $a+xI_n$ is non-invertible for finitely many elements $x\in F$: if $a+xI_n$ is non-invertible, then the determinate $|a+xI_n|$ of $a+xI_n$, a polynomial of degree $\le n$ in $x$, is $0$. By the Vandermonde argument \cite[Propositions 2.3.26 and 2.3.27]{Bo_Ro_80}, there are finitely many elements $x\in F$ such that $|a+xI_n|=0$.

Now assume that $H$ is an almost subnormal subgroup of $G$, and
$$H=H_r\leq H_{r-1}\leq \cdots\leq H_1=G$$
is an almost normal series of subgroups of $G$. For any $a,b\in H$ and $x\in F$ such that $b+xI_n$ is invertible, put $c_1(a,b,x):=(b+xI_n)a(b+xI_n)^{-1}$, and for $1<i\le r$, we define $c_i$ inductively as the following: if $H_i$ is normal in $H_{i-1}$, then  $c_i(a,b,x):=c_{i-1}bc_{i-1}^{-1}$, otherwise $c_i(a,b,x):=c_{i-1}^{\ell_i}$, where $\ell_i$ is the index of $H_i$ in $H_{i-1}$.

\begin{Lemma}\label{l6.3} Let $c_i(a,b,x)$ be as  above. Then, $c_i=(b+xI_n)w_i(a,b)(b+xI_n)^{-1}$, where $w_i(a,b)$ is a reduced word in $a,b,a^{-1},b^{-1}$ which begins and ends by $a$ or $a^{-1}$.
\end{Lemma}
\begin{proof}
	We prove the lemma by induction in $1\le i\le r$. If $i=1$, then $w_1(a,b)=(b+xI_n)^{-1}c_1(b+xI_n)=a$. Assume that $c_i=(b+xI_n)w_i(a,b)(b+xI_n)^{-1}$, where  a reduced word $w_i(a,b)$ begins and ends by $a$ or $a^{-1}$ for any $i\le 1$. We have to show  that $c_{i+1}$ has a same property, that is, $c_{i+1}(a,b,x)=(b+xI_n)w_{i+1}(a,b)(b+xI_n)^{-1}$ with a reduced word $w_{i+1}(a,b)$ ending by $a$ or $a^{-1}$. Indeed, there are two cases to examine.
	
	\bigskip
	
	\noindent
	{\it Case 1}. If $H_{i+1}$ is normal in $H_i$, then
	
	\hspace*{0.5cm} $c_{i+1}(a,b,x)=c_ibc_i^{-1}$
	
	\hspace*{0.5cm} $=((b+xI_n)w_i(a,b)(b+xI_n)^{-1})b((b+xI_n)w_i(a,b)(b+xI_n)^{-1})^{-1}$
	
	\hspace*{0.5cm} $=(b+xI_n)w_i(a,b)bw_i(a,b)^{-1}(b+xI_n)^{-1}$
	
	\hspace*{0.5cm} $=(b+xI_n)w_{i+1}(a,b)(b+xI_n)^{-1},$ where $w_{i+1}=w_i(a,b)bw_i(a,b)^{-1}$.
	
\bigskip	

\noindent
	{\it Case 2}. If $H_{i+1}$ has finite index $r_{i+1}$ in $H_i$, then
	
	\hspace*{0.5cm} $c_{i+1}(a,b,x)$
	
	\hspace*{0.5cm} $=c_i(a,b,x)^{\ell_{i+1}}$
	
	\hspace*{0.5cm} $=((b+xI_n)w_i(a,b)(b+xI_n)^{-1})^{\ell_{i+1}}$
	
	\hspace*{0.5cm} $=(b+xI_n)w_i(a,b)^{\ell_{i+1}}(b+xI_n)^{-1}$
	
	\hspace*{0.5cm} $=(b+xI_n)w_{i+1}(a,b)(b+xI_n)^{-1}$, where $w_{i+1}=w_i(a,b)^{\ell_{i+1}}$.
	
	The proof of the lemma is now complete.	
\end{proof}

\bigskip 

Notice that in the proof of main results in \cite{Pa_MaMaYa_00}, the authors considered  two cases $n=1$ and $n>1$ separately with two difference arguments. By modifying the proof of the case $n=1$ in \cite{Pa_MaMaYa_00}, we will prove our main result as the following for the arbitrary case.

\begin{theorem}\label{t6.4}
	Let $D$ be a weakly locally finite division ring, and assume that $G$ is an infinite almost subnormal subgroup of $\GL_n(D)$. If $G$ is finitely generated, then $G$ is central.
\end{theorem}
\begin{proof} Assume by contrary that $G$ is non-central finitely generated subgroup of $\GL_n(D)$. Since $D$ is weakly locally finite, by Remark~\ref{r6.2},  without loss of generality, we may assume that $D$ is centrally finite with $[D:F]=t<\infty$, where $F$ is the center of $D$. 
	
	We first claim that $G$ contains a non-cyclic free subgroup by showing that $D$ and $G$ satisfy the requirements  of Theorem \ref{t4.3'}. Indeed, if $\Char(F)=p>0$ and $D=F$ is a field algebraic over its prime subfield, then every element of $\GL_n(F)$ is torsion. By Schur's Theorem \cite[Theorem (9.9), p. 154]{Bo_La_91}, $G$ is finite that contradicts to the hypothesis.  Therefore, the claim is proved.
	
	Clearly the group $\GL_n(D)$ may be viewed as a subgroup of $\GL_{nt}(F)$, so $G$ is also a subgroup of $\GL_{nt}(F)$. By Remark~\ref{r6.2}, $G$ is a subgroup of $\GL_{nt}(P(S))$, where $P$ is the prime subfield of $F$ and $S$ is a finite set of $F$. Since $G$ is infinite, so is $P(S)$. Let us consider two cases.
\bigskip

\noindent	
	{\it Case 1}. $\Char(P)>0$: 
	
Let $S=\{\alpha_1, \alpha_2, \ldots, \alpha_h \}$. If all elements of $S$ are algebraic over $P$, then $P(S)$ is finite, a contradiction. Let $i_0$ be the largest index such that $\alpha:=\alpha_{i_0}$ is not algebraic over $P$. If $K=P(\alpha_1, \alpha_2, \ldots, \alpha_{i_0-1})$ and $L=K(\alpha)$, then $[P(S):L]=s<\infty$.	
	Therefore, $G$ may be viewed as a subgroup of $\GL_{nts}(L)$. Recall that $\alpha$ is not algebraic over $K$, so $L$ can be considered as the field of fractions of $K[\alpha]$. Again by Remark~\ref{r6.2}, we can find a finite subset $$T=\left\{\, \frac{u_1(\alpha)}{v_1(\alpha)},\frac{u_2(\alpha)}{v_2(\alpha)},\cdots, \frac{u_m(\alpha)}{v_m(\alpha)}\,\right\}$$ of $L$ such that $\GL_{nts}(K[\alpha][T])$ contains $G$. Note that $v_1(\alpha),u_1(\alpha),\cdots, v_m(\alpha),u_m(\alpha)$ are elements of $K[\alpha]$ such that $v_i(\alpha)$ and $u_i(\alpha)$ are co-prime for any $i$. Let $a,b\in G$ be two elements such that $\langle a,b\rangle$ is a non-abelian free subgroup of $G$ and let $G=G_r\subseteq G_{r-1}\subseteq \cdots\subseteq G_1=\GL_n(D)$ be an almost normal series of $G$ in $\GL_{nts}(K[\alpha][T])$. Observe that $b+xI_{nts}$ is non-invertible for finitely many elements $x\in K[\alpha][T]$. Now, take an  $x\in K[\alpha][T]$ such that $b+xI_{nt}$ is invertible. By Lemma~\ref{l6.3}, $c_r(a,b,x)=(b+xI_{nts})w_r(a,b)(b+xI_{nts})^{-1}\in G_r=G$. We claim that all the entries of $c_r(a,b,x)$ do not depend on $x$. Indeed, assume that there exists an entry $(i,j)$-th of $c_r(a,b,x)$ depending on $x$. Without loss of generality, we may assume that  $(i,j)=(1,1)$. Notice that the determinant of $b+xI_{nts}$ is a polynomial $f(x)$ in $x$ of degree $q=nts$, so the $(1,1)$-th entry has the form $$\frac{g(x)}{f(x)}=\frac{b_qx^q+b_{q-1}x^{q-1}+\cdots +b_0}{x^q+c_{q-1}x^{q-1}+\cdots+c_0}\in K[\alpha][T].$$ Assume that $b_q=\frac{u_{m+1}(a)}{v_{m+1}(a)}$. Then,  $\frac{g(x)}{f(x)}-b_q\in K[\alpha][T\cup\{\frac{u_{m+1}(a)}{v_{m+1}(a)}\}]$, and clearly, the degree of the numerator of $\frac{g(x)}{f(x)}-b_q$ is less than $q$. Hence, without loss of generality, we may assume that $b_q=0$, that is, $$\frac{g(x)}{f(x)}=\frac{b_{q-1}x^{q-1}+\cdots+b_0}{x^{q}+c_{q-1}x^{q-1}+\cdots+c_0}\in K[\alpha][T].$$
	Observe that $c_0$ is the determinant of $b$, which is invertible element in $\GL_{nts}(K[\alpha][T])$, so  $c_0\ne 0$.
	Put $$\frac{g_1(x)}{f_1(x)}=\frac{c_0^{-1}g(x)}{c_0^{-1}f(x)},$$ that is, $g_1(x)=c_0^{-1}b_{q-1}x^{q-1}+\cdots+c_0^{-1}b_0$ and $f_1(x)=c_0^{-1}x^{q}+c_0^{-1}c_{q-1}x^{q-1}+\cdots+1$. Let $w_1(\alpha),\cdots, w_\ell(\alpha)$ be  all prime factors of $v_1(\alpha),u_1(\alpha), \cdots, v_m(\alpha),u_m(\alpha)$ and put $$x(\alpha)=(w_1(\alpha)w_2(\alpha)\cdots w_l(\alpha))^p.$$ Then, since the degree $g_1(x)$ is less than $f_1(x)$'s when $p$ is large enough, the degree of the denominator $f_1(x(\alpha))$ in $\alpha$ is greater than the numerator's which implies that there exists $i$ such that $f_1(x(\alpha))$ is a multiple of $w_i(\alpha)$. Hence, $1$ is also a multiple of $w_i(\alpha)$ which is a contradiction. Thus, the claim is proved. Therefore, $c_r(a,b,x)$ does not depend on $x$. Now, one has  $c_r(a,b,0)=c_r(a,b,y)$ for some $y\in K[\alpha]\backslash\{0\}$ such that $b+yI_q$ is invertible. Hence, $bw_r(a,b)b^{-1}=(b+yI_q)w_r(a,b)(b+yI_q)$, equivalently, $bw_r(a,b)b^{-1}=w_r'(a,b)$, which is a contradiction to the fact that $a,b$ are generators of a non-cyclic free group.
\bigskip

\noindent
{\it Case 2}. $\Char(F)=0$:
	
	 If $P(S)$ is not algebraic over $\Q$, then by the same procedure as in the first part of {\it Case 1}, we conclude that   the field $P(S)$ contains a subfield $L_1=K_1(\beta)$ such that $[P(S):L_1]=s_1<\infty$, where $K_1$ is a subfield of $P(S)$ and $\beta$ is not algebraic over $K_1$. Now, again by the same procedure as in Case 1 with replacing $K(\alpha)$ by $K_1(\beta)$,  one has a contradiction.
	 
	 Therefore, the proof of the theorem is now complete.
	
\end{proof}


\bigskip  


\begin{thebibliography}{99}
\bibitem{artin} E. Artin, \textit{Geometric Algebra}, Interscience Publishers, Inc., New York, Interscience Publishers Lmd, London, 1957.
\bibitem{bach} S. Bachsmuth, In {\it ``Proceedings of the Second International Conference on the Theory of Groups"}
(Canberra, Australia, 1973), Lecture Notes in Math., Vol. \textbf{372}, Springer-Verlag,
Berlin and New York (1974). MR $\#$ 9054.
\bibitem{Pa_Bi_15} M. H. Bien,  On some subgroups of $D^*$ which satisfy a generalized group identity, {\it Bull. Korean Math. Soc.} \textbf{52} (2015), No. 4, pp. 1353--1363.
\bibitem{Pa_BiDu_14} M. H. Bien, D. H. Dung,  On normal subgroups of division rings which are radical over a proper division subring, {\it Studia Sci. Math. Hungar.} \textbf{51} (2014), no. 2, 231--242.

\bibitem{Pa_BiKiRa_?} M. H. Bien, D. Kiani, M. Ramezan-Nassab,  Some skew linear groups satisfying generalized group identities, {\it Comm. Algebra.} DOI: 10.1080/00927872.2015.1044109. 

\bibitem{Pa_BeDrSh_13} J. P. Bell, V. Drensky, Y. Sharifi, Shirshov's theorem and division rings that are left algebraic over a subfield, \textit{J. Pure Appl. Algebra} \textbf{217} (2013), no. 9, 1605--1610.

\bibitem{chiba} K. Chiba, Free subgroups and free subsemigroups of division rings, \textit{J. Algebra} \textbf{184} (1996), no. 2, 570--574. 
\bibitem{Pa_DeBiHa_12} T. T. Deo, M. H. Bien, B. X. Hai, On radicality of maximal subgroups in ${\rm GL}_n(D)$, \textit{J. Algebra} \textbf{365} (2012)  42--49.

\bibitem{Pa_GoMi_82} I. Z. Golubchik and A. V. Mikhalev, Generalized group identities in the classical groups, \textit{Zap. Nauch. Semin. LOMI AN SSSR} \textbf{114} (1982): 96--119.
\bibitem{gon_84_1}J. Z. Gon\c calves, Free subgroups of units in group rings, \textit{Canad. Math. Bull.}, \textbf{27} (1984), 309–312.
\bibitem{gon_84_2} J. Z. Gonc¸alves, Free groups in subnormal subgroups and the residual nilpotence
of the group of units of group rings, \textit{Can. Math. Bull.}, \textbf{27} (1984), 365–370.
\bibitem{gon-ma_86} J. Z. Gon\c calves, A. Mandel, Are there free groups in division rings? \textit{Israel J. Math.} \textbf{53} (1986), no. 1, 69--80.
\bibitem{Pa_GoPa_15} J. Z. Gon\c calves, D. S. Passman, Free groups in normal subgroups of the multiplicative group of a division ring, \textit{J. Algebra} \textbf{440} (2015), 128--144.
\bibitem{Pa_GoSh_09} J. Z. Gon\c calves, M. Shirvani, Algebraic elements as free factors in simple Artinian rings, \textit{Contemp. Math.} \textbf{499}, 121--125.
\bibitem{Pa_HaDeBi_12} B. X. Hai, T. T. Deo and M. H. Bien, On subgroups of division rings of type $2$, \textit{Studia Sci. Math. Hungar.} \textbf{49} (2012), no. 4, 549--557.
\bibitem{hai-ngoc} B. X. Hai, N. K. Ngoc, A note on the existence of non-cyclic free subgroups in division rings, {\em Archiv der Math.} {\bf 101} (2013), 437-443.
\bibitem{hbd} B. X. Hai, M. H. Bien, and T. T. Deo, On the Gelfand-Kirillov dimension of weakly locally finite division rings, \textit{arXiv:1510.08711 v1[math.RA]} 29 Oct 2015.
\bibitem{Pa_Ha_89} B. Hartley, Free groups in normal subgroups of unit groups and arithmetic groups, \textit{Contemp. Math.} \textbf{93} (1989) 173--177.
\bibitem{Haz-Wad} R. Hazrat and A.R. Wadsworth, On maximal subgroups of the multiplicative group of a division algebra, {\em J. Algebra} {\bf 322}, 7 (2009), 2528 - 2543.
\bibitem{Pa_HaMaMo_14} R. Hazrat, M. Mahdavi-Hezavehi, M. Motiee, Multiplicative groups of division rings, \textit{Math. Proc. R. Ir. Acad.} \textbf{114A} (2014), 37--114.

\bibitem{Bo_La_91}  T. Y. Lam,   \textit{A First Course in Noncommutative Rings}, GMT  {\bf 131}, Springer, 1991.
\bibitem{lichtman_77} A. I. Lichtman, On subgroups of the multiplicative group of skew fields, \textit{Proc. Amer. Math. Soc.} \textbf{63} (1977), 15–16.
\bibitem{Pa_Li_78} A. I. Lichtman, Free subgroups of normal subgroups of the multiplicative group of skew fields, \textit{Proc. Amer. Math. Soc.} \textbf{71} (1978), no. 2, 174--178.
\bibitem{lichtman_87} A. I. Lichtman, Free subgroups in linear groups over some skew fields, \textit{J. Algebra} \textbf{105} (1987), 1-28.

\bibitem{Pa_MaMaYa_00} M. Mahdavi-Hezavehi, M. G. Mahmudi, M. G, S. Yasamin, Finitely generated subnormal subgroups of $\GL_n(D)$ are central, \textit{J. Algebra} \textbf{225} (2000), no. 2,  517--521.
\bibitem{tits} J. Tits, Free subgroups in linear groups, \textit{J. Algebra} {\textbf{20}} (1972), 250--270.
\bibitem{Pa_To_85} G. M. Tomanov,  Generalized group identities in linear groups, \textit{Math.  Sbornik} Vol. \textbf {51} (1985), 33--46.
\bibitem{rei-von} Z. Reichstein, N. Vonessen, Free subgroups of division algebras, \textit{Comm. Algebra} \textbf{23} (1995), no. 6, 2181--2185.

\bibitem{Bo_Ro_80} L. H. Rowen, \textit{Polynomial identities in ring theory}, Academic, 1980.

\bibitem{shir-weh}  M. Shirvani  and  B. A. F. Wehrfritz, {\em Skew Linear Groups}, Cambridge Univ. Press, Cambridge, 1986.

\end{thebibliography}
\end{document}